\documentclass[12pt]{amsart}
\usepackage{amsmath}
\usepackage{amssymb}
\usepackage{amsthm}
\usepackage{amscd}
\usepackage{enumerate}
\usepackage{verbatim}
\usepackage[all]{xy}
\usepackage{mathrsfs}
\usepackage{url}

\theoremstyle{plain}
\newtheorem{thm}{Theorem}[section]
\newtheorem{prop}[thm]{Proposition}
\newtheorem{lem}[thm]{Lemma}
\newtheorem{cor}[thm]{Corollary}

\theoremstyle{definition}
\newtheorem{dfn}[thm]{Definition}
\newtheorem{rmk}[thm]{Remark}

\newcommand{\Hom}{\mathrm{Hom}}

\newcommand{\Spf}{\mathrm{Spf}}
\newcommand{\Spv}{\mathrm{Sp}}
\newcommand{\Spec}{\mathrm{Spec}}

\newcommand{\Gal}{\mathrm{Gal}}

\newcommand{\rig}{\mathrm{rig}}

\newcommand{\ord}{\mathrm{ord}}

\newcommand{\red}{\mathrm{red}}

\newcommand{\okey}{\mathcal{O}_K}

\newcommand{\cC}{\mathcal{C}}
\newcommand{\cD}{\mathcal{D}}

\newcommand{\cH}{\mathcal{H}}

\newcommand{\cO}{\mathcal{O}}

\newcommand{\cW}{\mathcal{W}}
\newcommand{\cX}{\mathcal{X}}
\newcommand{\cY}{\mathcal{Y}}

\newcommand{\Gm}{\mathbb{G}_\mathrm{m}}

\newcommand{\Cp}{\mathbb{C}_p}

\newcommand{\bA}{\mathbb{A}}

\newcommand{\bC}{\mathbb{C}}

\newcommand{\bQ}{\mathbb{Q}}

\newcommand{\bZ}{\mathbb{Z}}

\newcommand{\frX}{\mathfrak{X}}

\makeatletter
    
    \@addtoreset{equation}{section}
\makeatother

\makeatletter
\renewcommand{\p@enumii}{}
\makeatother

\begin{document}

\title[Irreducible components of the eigencurve]{Irreducible components of the eigencurve of finite degree are finite over the weight space}
\author{Shin Hattori}
\address[Shin Hattori]{Faculty of Mathematics, Kyushu University}
\author{James Newton}
\address[James Newton]{Department of Mathematics, King's College London}
\date{\today}


\begin{abstract}
Let $p$ be a rational prime and $N$ a positive integer which is prime to $p$. 
Let $\cW$ be the $p$-adic weight space for $\mathit{GL}_{2,\bQ}$. Let $\cC_N$ 
be the $p$-adic Coleman-Mazur eigencurve of tame level $N$. In this paper, we 
prove that any irreducible component of $\cC_N$ which is of finite degree over 
$\cW$ is in fact finite over $\cW$. 

Combined with an argument of Chenevier and 
a conjecture of Coleman-Mazur-Buzzard-Kilford (which has been proven in special cases, and 
for general quaternionic eigencurves) this shows that the only finite degree 
components of the eigencurve are the ordinary components.
\end{abstract}

\maketitle



\section{Introduction}

Let $p$ be a rational prime. Let $\bar{\bQ}_p$ be an algebraic closure of $\bQ_p$ and $\Cp$ the $p$-adic completion of $\bar{\bQ}_p$. For any element $a\in \Cp$, we denote its $p$-adic valuation by $v_p(a)$, which we normalize as $v_p(p)=1$, and its $p$-adic absolute value by $|a|_p=p^{-v_p(a)}$. 
Let $\cW$ be the weight space for $\mathit{GL}_{2,\bQ}$ over $\bQ_p$ satisfying
\[
\cW(\Cp)=\Hom_{\mathrm{cont.}}(\bZ_p^\times,\cO_{\Cp}^\times).
\]

For any positive integer $N$ which is prime to $p$, we have a reduced rigid 
analytic curve $\cC_N$ over $\cW$, the Coleman-Mazur eigencurve, which 
parametrizes the overconvergent eigenforms of tame level $N$ and finite slope 
\cite{CM, Buz}. Though it plays an important role in number theory, the global 
geometry of $\cC_N$ is not well understood. However, with recent groundbreaking 
works including \cite{DL,LWX}, we are now gradually revealing what $\cC_N$ 
looks like globally.

Let $C$ be an irreducible component of $\cC_N$. We know by Hida theory that $C$ is finite over $\cW$ if it is in the ordinary locus $\cC_N^\ord$. Conversely, Chenevier proved that if $C$ is finite over $\cW$, then a conjecture of Coleman-Mazur-Buzzard-Kilford (see \cite[Conjecture 1.2]{LWX}), claiming that the slopes should tend to zero near the boundary of $\cW$, implies that $C$ is in $\cC_N^\ord$ \cite[Chapitre 1, \S 3.7]{Che}. Note that the conjecture is now known to be true in many cases \cite{BuzKil,LWX,Roe}.

An irreducible component $C$ of $\cC_N$ is said to be of finite degree if the 
degree of any fiber of the weight map $C \to \cW$ is finite (we do not impose 
the requirement that the map from $C \to \cW$ is quasi-compact, so this is a 
more general notion than a component which is quasi-finite over the weight space). 
Otherwise, it is 
said to be of infinite degree. 
In \cite[\textit{loc. cit.}]{Che}, Chenevier predicts that any irreducible 
component of finite degree should be finite over $\cW$, and that this should 
follow from the \textit{properness} of $\cC_N$ in the sense of Buzzard-Calegari 
\cite{BuzCal}. The latter property means that any morphism from the punctured 
closed unit disc to $\cC_N$ extends to the puncture, and Diao-Liu proved that 
$\cC_N$ is proper in this sense, at least when the morphism is defined over a 
finite extension of $\bQ_p$ \cite[Theorem 1.1]{DL}. In short, what Chenevier 
expected can be viewed as a version of `quasi-finite plus proper implies 
finite'.

The problem is that this properness is a quite different notion from 
the properness in the usual sense in rigid analytic geometry. In fact, since $\cC_N$ is finite over $\cW\times \Gm$ and thus it is partially proper over $\cW$, the weight map $\cC_N\to \cW$ is proper in the usual sense if and only if it is quasi-compact. However,
since the slopes of classical cuspforms of level $\Gamma_1(Np)$ and weight $k$ can be arbitrarily large as $k$ increases, 
the map $\cC_N\to \cW$ is never quasi-compact. Moreover, the partial properness over $\cW$ does not necessarily imply the properness in the above sense; for example, the natural open immersion from the punctured open unit disc $A(0,1)$ to any of the connected components of $\cW$ is partially proper, while $A(0,1)$ does not have the extension property as above.
In this paper, we nonetheless prove that Chenevier's prediction is true, as follows.

\begin{thm}\label{main}
Let $C$ be an irreducible component of the eigencurve $\cC_N$ which is of 
finite degree. Then $C$ is finite over $\cW$.
\end{thm}

By combining the theorem with known cases of the Coleman-Mazur-Buzzard-Kilford conjecture, 
we also prove that irreducible components of eigencurves of finite degree are 
in the ordinary loci in many cases (see \S\ref{SecApp}), which gives a partial answer to a question raised by Coleman-Mazur \cite[p. 5]{CM}.

We end this introduction by sketching the proof of Theorem \ref{main}. We may 
assume that $C$ is a non-Eisenstein component. If $C$ contains classical 
oldforms as a Zariski dense subset, then we will relate $C$ to an irreducible 
component of the eigencurve of a smaller tame level via degeneracy maps (Lemma 
\ref{qTU}). By an induction on the tame level, this will reduce ourselves to 
the case where $C$ contains enough classical newforms. 

In this case, we will show that the projection to a spectral curve induces a 
`generic isomorphism' from $C$ to a Fredholm hypersurface $Y$ of finite degree 
(Lemma \ref{genisom}). A similar result for $N=1$ is stated in \cite[Corollary 
7.6.2]{CM}, while what we will prove is slightly weaker: there exists a finite 
set $F_0 \subseteq Y$, discrete subsets $E\subseteq Y\setminus F_0$ and 
$E'\subseteq (Y\setminus F_0)\setminus E$ such that $C\to Y$ is isomorphic over 
$((Y\setminus F_0)\setminus E)\setminus E'$.

By the $p$-adic Riemann existence theorem \cite{Lut,Ram}, 
for any $x\in \cW$ in the component to which $C$ is mapped, there exists a 
sufficiently small closed punctured disc $D^\times$ centered at $x$ such that 
the covering $Y|_{D^\times}\to D^\times$ is of Kummer type over $\bC_p$. 
Namely, it is isomorphic over $\bC_p$ to the direct sum of finitely many 
punctured closed discs centered at the origin and the projection from each 
connected component to $D^\times$ is given by the map $z\mapsto x+z^m$ with 
some positive integer $m$.
Since $Y$ is a Fredholm hypersurface over $\cW$, we can descend this isomorphism to a finite extension of $\bQ_p$ (Proposition \ref{Lutke}).

Hence $((Y\setminus F_0)\setminus E)\setminus E'$ is locally isomorphic to a 
punctured closed disc minus a discrete subset minus a discrete subset, and the 
above `generic isomorphism' gives a section from it to $C$. Now Diao-Liu's 
properness of the eigencurve enables us to fill the holes of discrete subsets 
and extend the section to the whole closed disc, which is enough to show the 
finiteness of $C$ (Lemma \ref{Yfinite}).

\subsection*{Acknowledgments}

S.H.~would like to thank Ga\"{e}tan Chenevier and Lorenzo Ramero for kindly 
answering his questions on their works. J.N.~would like to thank Rebecca Bellovin, Kevin Buzzard 
and Olivier Ta\"{i}bi for helpful conversations. Both authors would like to thank the author of the 
blog Persiflage whose post \cite{Persblog} initiated our interest in the 
question answered by this paper.
S.H.~was supported by JSPS KAKENHI Grant Number 26400016 and J.N.~was supported by ERC Starting Grant 306326. 



\section{Fredholm hypersurfaces}

Let $K/\bQ_p$ be a finite extension in $\bar{\bQ}_p$. Consider the formal power series ring $\cO_K [[w]]$ and put $\cW^0=\Spf(\cO_K[[w]])^\rig$. Let $Y$ be a Fredholm hypersurface in $\cW^0\times \bA^1$ \cite[\S 1.3]{CM} and $\mu:Y\to \cW^0$ the natural projection. We denote the parameter for $\bA^1$ by $T$. Let 
\[
Q(T)=1+ a_1(w) T + \cdots + a_n(w) T^n+\cdots
\]
be the Fredholm power series with coefficients in $\cO_K[[w]]$ defining $Y$. Note that, by the Weierstrass preparation theorem, every non-zero $a_n(w)$ has only finitely many zeros in $\cW^0$. If $a_n(w)\neq 0$ for infinitely many $n$, then the set 
\[
\bigcup_{m\geq 1}\bigcap_{n\geq m}\{x\in \cW^0\mid a_n(x)=0\}
\]
is countable and there exists $x\in \cW^0$ such that $\mu^{-1}(x)$ is an infinite set. In this case, we say $Y$ is of infinite degree. Otherwise, there exists an integer $n$ satisfying $a_m(w)=0$ for any $m>n$ and $a_n(w)\neq 0$. Then, for any $x\in \cW^0$, the degree of the fiber $\mu^{-1}(x)$ is bounded by $n$.
In this case, we say $Y$ is of finite degree. In the finite degree case, put
\begin{equation}\label{eqnB}
B=\cO_K[[w]][U]/(U^n+ a_1(w) U^{n-1} + \cdots + a_n(w))
\end{equation}
and $X=\Spf(B)^\rig$, which is the projective closure of $Y$ over $\cW^0$. Then 
$X$ is finite flat over $\cW^0$ and $Y$ is identified with the Zariski open 
subspace $X[1/U]$ of $X$ via $U=T^{-1}$. Note that $Y$ is the complement of a 
subset of the zeros of $a_n(w)$ in $X$. In both finite and infinite degree 
cases, the image of $Y$ in $\cW^0$ is the complement of a finite set.

\begin{lem}\label{Yfinite}
Let $D$ be an admissible affinoid open subset of $\cW^0$ and put $Y|_D=Y\times_{\cW^0} D$. Suppose that there exist a $K$-affinoid variety $S$ which is finite over $D$ and a morphism $f: S\to Y|_D$ over $D$ such that $Y|_D\setminus f(S)$ is a finite set. Then $Y|_D$ is finite over $D$.
\end{lem}
\begin{proof}
Let $T$ be the parameter of $\bA^1$ as above. This defines a rigid analytic function $T|_{Y|_D}$ on $Y|_{D}$. Since $S$ is affinoid, the pull-back of $T|_{Y|_{D}}$ to $S$ is bounded and the supremum semi-norm of $T|_{Y|_{D}}$ on the subset $f(S)\subseteq Y|_{D}$ is also bounded. Since the complement of $f(S)$ is a finite set, $T|_{Y|_{D}}$ is bounded on $Y|_{D}$. This means that the natural closed immersion $Y|_{D}\to D\times \bA^1$ factors through $D\times B[0,r]$ with the closed ball $B[0,r]$ of some radius $r$ centered at the origin. Thus $Y|_{D}$ is affinoid. 

Consider the nilreduction $(Y|_D)_\red$ of $Y|_D$. Since $f(S)\subseteq Y|_D$ is Zariski dense, the affinoid map $\cO((Y|_D)_\red)\to \cO(S_\red)$ is an injection. Since $\cO(S_\red)$ is finite over $\cO(D)$, so is $\cO((Y|_D)_\red)$. Hence $(Y|_D)_\red$ and $Y|_D$ are finite over $D$. 
\end{proof}

\begin{lem}\label{finet}
Suppose that $Y$ is of finite degree and reduced. Then there exists a finite subset $F$ of $\cW^0$ such that the map $Y\setminus \mu^{-1}(F) \to \cW^0\setminus F$ is finite etale. 
\end{lem}
\begin{proof}
Let $B$ be as in (\ref{eqnB}) and $I$ the discriminant ideal of the finite 
locally free map $A=\cO_K[[w]]\to B$. Then $I$ is a locally principal ideal of 
$A$. Since $A$ is a UFD, the ideal $I$ is principal. By the Weierstrass 
preparation theorem, it is enough to show $I\otimes_{\okey} K\neq 0$ in the 
ring $A\otimes_{\cO_K}K$. Note that for any prime ideal $\mathfrak{q}$ of 
$B\otimes_{\cO_K}K$ above $(0)$, the ring $(B\otimes_{\cO_K}K)/\mathfrak{q}$ is 
a finite extension of $A\otimes_{\cO_K}K$ and thus contains infinitely many 
maximal ideals. From \cite[Lemma 7.1.9]{deJong} and the assumption, we see that 
$\Spec(B\otimes_{\okey} K)$ is reduced outside a finite set of closed points. 
This implies that
the map $\Spec(B\otimes_{\okey} K)\to \Spec(A\otimes_{\okey} K)$ is etale over 
the generic point of the target. Hence we obtain $I\otimes_{\okey} K\neq 0$.
\end{proof}

For any quasi-separated rigid analytic variety $\cX$ over $K$ and any extension $L/K$ of complete valuation fields, we write $\cX_{L}=\cX\hat{\otimes}_{K}L$. 

\begin{prop}\label{Lutke}
Suppose that $Y$ is of finite degree and reduced. Let $x$ be any element of $\cW^0(K)$. Then there exists a closed disc $D$ of sufficiently small radius in $|K^\times|_p$ centered at $x$ such that $D$ is an admissible open subset of $\cW^0$ and, for the punctured disc $D^\times=D\setminus \{x\}$ and some finite extension $L/K$, the covering $Y_L |_{D^\times_L} \to D^\times_L$ is of Kummer type. Namely, there exists an isomorphism
\[
\coprod_{i=1,\ldots,M} D_{i,L}^\times \to Y_L |_{D_L^\times}
\]
over $D_L^\times$, where $D_{i,L}$ is a closed disc of some radius over $L$ centered at the origin $O$, $D^\times_{i,L}=D_{i,L}\setminus \{O\}$ and the map $D_{i,L}^\times \to D_{L}^\times$ is the finite flat surjection defined by $z\mapsto x+z^{m_i}$ for some positive integer $m_i$. 
\end{prop}
\begin{proof}
Let $X$ be the projective closure of $Y$ as above. Note that $X\setminus Y$ is a finite set. By Lemma \ref{finet}, after shrinking the radius of $D$, we may assume $Y|_{D^\times}=X|_{D^\times}$ and the map $\mu: Y|_{D^\times} \to D^\times$ is finite etale. Then, the $p$-adic Riemann existence theorem (see \cite[Proof of Corollary 2.9]{Lut} and \cite[\S 2.4]{Ram}) implies that, after shrinking the radius of $D$ further,
the restriction $Y_{\Cp}|_{D_{\Cp}^\times}$ is isomorphic to the disjoint union 
\[
\coprod_{i=1,\ldots,M} D_{i,\Cp}^\times 
\]
of punctured discs $D_{i,\Cp}^\times$ for closed discs $D_{i,\Cp}$ of some 
radii centered at the origin and, on the component $Y_{\Cp,i}^\times$ of 
$Y_{\Cp}|_{D_{\Cp}^\times}$ corresponding to $D_{i,\Cp}^\times$, the map 
$\mu_{\Cp}: Y^\times_{\Cp,i}\to D_{\Cp}^\times$ is identified with the finite 
flat surjection $D^\times_{i,\Cp} \to D^\times_{\Cp}$ defined by $z\mapsto 
x+z^{m_i}$ for some positive integer $m_i$. In particular, after replacing $K$ 
with a finite extension, we may assume that $D_{i,\Cp}$ is the base extension 
of a closed disc $D_i$ over $K$ with radius in $|K^\times|_p$. Then 
$D_{i,\Cp}^\times$ is also the base extension to $\Cp$ of the punctured disc 
$D_{i}^\times$ over $K$. It now remains to show that the maps $D_{i,\Cp}^\times 
\rightarrow Y_{\Cp}|_{D_{\Cp}^\times}$ are defined over a finite extension $L$ 
of $K$. To do this, we use the following lemma:

\begin{lem}\label{intcl}
The integral closure of the Tate algebra $K\langle t \rangle$ in $\Cp\langle t \rangle$ is
\[
\bigcup_{L/K} L\langle t \rangle,
\]
where the union is taken over the set of finite extensions $L/K$ in $\Cp$.
\end{lem}
\begin{proof}
Any element of the set of the lemma is integral over $K\langle t \rangle$. Conversely, let $f$ be a root in $\Cp\langle t \rangle$ of a monic polynomial over $K\langle t \rangle$. Then the absolute Galois group $\Gal(\bar{\bQ}_p/K)$ acts continuously on the finite set of roots in $\Cp\langle t \rangle$ of this polynomial, and thus the action factors through the quotient by a subgroup $\Gal(\bar{\bQ}_p/L)$ with some finite extension $L/K$. This yields $f\in \Cp\langle t \rangle^{\Gal(\bar{\bQ}_p/L)}=L\langle t \rangle$.
\end{proof}

As promised, we now show that the map $D_{i,\Cp}^\times \to 
Y_{\Cp}|_{D^\times_{\Cp}}$ is defined over a finite extension $L$ of $K$. 
Indeed, consider the composite
\[
D_{i,\Cp}^\times \to Y_{\Cp}|_{D^\times_{\Cp}}\to X_{\Cp}|_{D_{\Cp}}.
\]
Since $x\in \cW^0(K)$, we can write $D=\Spv(K\langle u \rangle)$ and $D_i=\Spv(K\langle t \rangle)$ so that the map $D^\times_{i,\Cp}\to D^\times_{\Cp}$ is given by $u\mapsto t^{m_i}$. Since the inclusion $D_{\Cp} \to \cW^0_{\Cp}$ is defined over $K$, we have
\[
X_{\Cp}|_{D_{\Cp}}=\Spv\left(\Cp\langle u \rangle [U]/(U^n+ b_1(u) U^{n-1} + \cdots + b_n(u)) \right)
\]
with the image $b_i(u)\in K\langle u \rangle$ of $a_i(w)$. Since $X_{\Cp}|_{D_{\Cp}}$ is affinoid, the function $U$ is bounded on $X_{\Cp}|_{D_{\Cp}}$. Thus the pull-back $f(t)\in \cO(D_{i,\Cp}^\times)$ of $U$ by the above composite is also bounded, and hence it is an element of the subring $\cO(D_{i,\Cp})=\Cp\langle t \rangle$. Since it is a root of the equation
\[
U^n+ b_1(t^{m_i}) U^{n-1} + \cdots + b_n(t^{m_i})=0,
\]
Lemma \ref{intcl} implies $f(t) \in L\langle t \rangle$ with some finite extension $L/K$. This means that the map $D_{i,\Cp}^\times \to Y_{\Cp}|_{D^\times_{\Cp}}$ is defined over $L$. Therefore, by \cite[Theorem A.2.4]{Con_MC}, we obtain an isomorphism
\[
\coprod_{i=1,\ldots,M} D_{i,L}^\times \to Y_L |_{D_L^\times}
\]
over $D_L^\times$ as in the proposition.
\end{proof}



\section{Eigencurves and their properness}

Let $\cD_N$ be the Coleman-Mazur-Buzzard eigencurve of tame level $N$ over the weight space $\cW$, which is constructed from the full Hecke algebra 
\[
\cH=\bZ_p[T_l\ (l\nmid pN),\ U_l\ (l\mid pN),\ \langle l\rangle_N\ (l\mid N)]
\] 
acting on the space of overconvergent modular forms of tame level $N$, via the eigenvariety machine \cite{Buz}. We denote by $\frX_{N}$ the subset of $\cD_N$ consisting of points corresponding to classical normalized eigenforms of level $\Gamma_1(Np^m)$ for some positive integer $m$ and of finite slope. We also denote by $\frX_N^0$ the subset of $\frX_N$ consisting of cuspidal eigenforms. By a similar argument to \cite[Proof of Proposition 3.5]{Che_JL}, we see that the subset $\frX_{N}$ is Zariski dense in $\cD_N$. 

We say an overconvergent modular form is cuspidal-overconvergent if the constant term of its Fourier expansion vanishes at any unramified cusp \cite[\S 3.6]{CM}.
By considering the action of $\cH$ on the space of cuspidal-overconvergent modular forms, we obtain the cuspidal eigencurve $\cD_N^0$. From \cite[Corollary 2.6]{Bel}, we see that $\frX_N^0$ is Zariski dense in $\cD^0_N$. We denote by $\frX_N^{\text{nc}}$ the subset of $\frX_N$ consisting of non-cuspidal-overconvergent forms.

For any $\alpha\in\cH$, we denote by $Z_{\alpha}$ the spectral curve for the operator $\alpha U_p$ \cite[\S 6.1]{CM}. Then the arguments in \cite[\S 7.3]{CM} are valid also for our case: we can define the isomorphism $\rho_{\beta,\alpha}$ in the proof of \cite[Proposition 7.3.5]{CM} over each admissible affinoid open subset of $\cW$ and glue them as in the proof of \cite[Lemma 5.6]{Buz}.
Thus, for any $\alpha\in 1+p\cH$, we have a finite map $\pi_{\alpha}:\cD_N \to Z_{\alpha}$. Note that, for any $\alpha \in 1+p\cH$ and any $x\in \cW$, every eigenvalue of $\alpha U_p$ appearing in the space of overconvergent modular forms of weight $x$ is an element of $\cO_{\Cp}$. This implies that the characteristic power series $P(T)$ of $\alpha U_p$ has coefficients with absolute values bounded by one, and on each connected component $\Spf(\bZ_p[[w]])^\rig$ of $\cW$, the power series $P(T)$ is an element of the ring $\bZ_p[[w]]\{\{T\}\}$.

We also have another construction of the nilreduction $\cD_{N,\red}$ of $\cD_N$ via Galois deformations, as follows. Let $\Sigma_{pN}$ be the set of places of $\bQ$ consisting of $\infty$ and the prime factors of $pN$. We denote by $G_{\bQ,\Sigma_{pN}}$ the Galois group of the maximal algebraic extension of $\bQ$ unramified outside $\Sigma_{pN}$. For any $p$-modular representation $\bar{V}$ of tame level $N$ \cite[\S 5.1]{CM}, we denote by $R_{\bar{V}}$ the universal deformation ring of the determinant of $\bar{V}$ \cite[Proposition 3.7]{Che_Psc} and put
\[
X_{\bar{V}}=\Spf(R_{\bar{V}})^\rig,\quad X_{N}=\coprod_{\bar{V}} X_{\bar{V}},
\]
where the disjoint union is taken over the set of isomorphism classes of such 
$\bar{V}$. By restricting the determinant of the universal pseudocharacter on 
the generic fiber to the inertia subgroup at $p$ and using local class field 
theory, we have a homomorphism $\bZ_p^\times \to \cO(X_N)^\times$. By twisting 
it as in \cite[\S 5.1]{CM}, we obtain a natural projection $X_N\to \cW$.

On the other hand, we have an injection 
\[
\frX_N \to X_N\times \Gm \times \prod_{l|N} \bA^1,\quad f\mapsto (\rho_f, a_p(f)^{-1}, (a_l(f))_l),
\]
where $\rho_f$ is the $p$-adic pseudocharacter of $G_{\bQ,\Sigma_{pN}}$ attached to $f$ and $a_l(f)$ is the eigenvalue of the Hecke operator $U_l$ acting on $f$. Then we denote by $\cC_N$ the Zariski closure of $\frX_N$ in the right-hand side. From the definition, the rigid analytic variety $\cC_N$ is reduced. Since we have a pseudocharacter of $G_{\bQ,\Sigma_{pN}}$ over $\cD_{N,\red}$ such that its specialization at $f$ is equal to $\rho_f$ for any $f\in \frX_N$, we have a natural map 
\[
\cD_{N,\red} \to X_N\times \Gm \times \prod_{l|N} \bA^1,
\]
which factors through $\cC_N$ by the Zariski density of $\frX_N\subseteq \cD_{N,\red}$. Then, by a similar argument to \cite[Proof of Lemma 2.4]{Che_Fern}, we see that the natural map
\[
\cD_{N,\red} \to X_N\times \cW\times \Gm \times \prod_{l|N} \bA^1
\]
is a closed immersion and the map $\cD_{N,\red} \to \cC_N$ is an isomorphism over $\cW$. 
By \cite[Lemma 5.8]{Buz}, $\cC_N$ is equidimensional of dimension one. Similarly, we have the cuspidal part $\cC_N^0\subseteq \cC_N$ obtained as the Zariski closure of the image of $\frX_N^0$ and an isomorphism $\cD_{N,\red}^0\simeq \cC_N^0$.

\begin{dfn}
Let $K/\bQ_p$ be an extension of complete valuation fields. Let $D_K$ be the closed unit disc over $K$ centered at the origin $O$ and put $D^\times_K=D_K\setminus \{O\}$. 
Let $\cX$ be a rigid analytic variety over $K$.
We say $\cX$ is Buzzard-Calegari proper at $K$-algebraic points if, for any finite extension $L/K$, any morphism $D_L^\times \to \cX_L$ over $L$ extends to a morphism $D_L \to \cX_L$.
\end{dfn}

Recall that an analytic subset of a rigid analytic variety $\cX$ is said to be discrete if it is equidimensional of dimension zero. Note that any analytic subset of a discrete subset of $\cX$ is also discrete, and an analytic subset of $\cX$ is discrete if and only if every irreducible component of it is discrete.

\begin{lem}\label{findisc}
Let $\varphi: \cX\to \cY$ be a finite morphism of rigid analytic varieties. Let 
$E\subseteq \cX$ and $F\subseteq \cY$ be discrete subsets. Then the analytic 
subsets $\varphi(E)$ and $\varphi^{-1}(F)$ are both discrete.
\end{lem}
\begin{proof}
We give $E$, $F$, $\varphi(E)$ and $\varphi^{-1}(F)$ the reduced structures. The induced map $\varphi_E:E \to \varphi(E)$ is finite surjective and both sides are reduced. For any admissible affinoid open subset $V=\Spv(A)$ of $\varphi(E)$, we denote by $B$ the affinoid ring of $\varphi^{-1}_E(V)$. Over $V$, the map $\varphi_E$ is given by a finite injection $A\to B$. Since $B$ is of dimension zero, so is $A$ and $\varphi(E)$ is discrete. Since $\varphi(\varphi^{-1}(F))\subseteq F$ is discrete, the same reasoning implies that $\varphi^{-1}(F)$ is discrete.
\end{proof}

Let $F$ be any discrete subset of $D^\times_K$. In the sequel, we will need to take a special type of covering of $D^\times_K$ separating the points in $F$, as follows. Take an admissible covering $D^\times_{K}=\bigcup_{j\in \bZ_{\geq 0}} U_{j}$ consisting of
\[
{U}_{j}=\{ y\in D_{K} \mid \rho'_{j} \leq|y|_p \leq \rho_{j}\} 
\]
with some positive real numbers $\rho_{j},\rho'_{j} \in |K^\times|_p$ satisfying 
\[
\rho'_{j} <\rho_{j+1} <\rho'_{j-1} <\rho_{j}
\]
for any $j$. Since $U_j$ is affinoid and $F$ is discrete, the set $F\cap U_j$ is finite and for some finite Galois extension $K_j/K$, every element of 
\[
(F\cap U_j)_{K_j}=\{y_{j,1},\ldots, y_{j,n_{j}}\}
\] 
is $K_j$-rational. For some sufficiently small positive real number $r_{j}\in |K_{j}^\times|_p$, the closed discs 
\[
D_{j,k}=\{y\in D_{K_{j}}\mid |y-y_{j,k}|_p \leq r_{j}\}
\]
are admissible open subsets of $U_{j, K_{j}}$ which are disjoint to each other.
For any positive real number $s_{j}\in |K_{j}^\times|_p$ satisfying $s_{j}<r_{j}$, put 
\[
V_{j}=\{ y\in {U}_{j,K_j}\mid |y-y_{j,k}|_p\geq s_{j}\text{ for any }k\},
\]
which is an admissible affinoid open subset of $D_{K_j}$.
Then
\[
{U}_{j,K_j}=\bigcup_{k=1,\ldots, n_{j}} {D}_{j,k} \cup V_{j}
\] 
is an admissible covering. We will refer to any covering of this type as a covering of $D^\times_K$ adapted to the discrete subset $F$.

\begin{lem}\label{PPdiscr}
Let $K/\bQ_p$ be an extension of complete valuation fields and $\cX$ a separated rigid analytic variety over $K$. Then $\cX$ is Buzzard-Calegari proper at $K$-algebraic points if and only if for any finite extension $L/K$ and any discrete subset $F$ of $D^\times_L$, any morphism $\varphi: D_L^\times\setminus F\to \cX_L$ over $L$ extends to a morphism $D_L \to \cX_L$.
\end{lem}
\begin{proof}
Suppose that $\cX$ is Buzzard-Calegari proper at $K$-algebraic points. We may assume $L=K$. Take a covering of $D^\times_K$ adapted to the discrete subset $F$ as above.
By assumption, each morphism
\[
D_{j,k}^\times =D_{j,k}\setminus\{y_{j,k}\}\to \cX_{K_j}
\]
induced by $\varphi$ extends to a morphism ${D}_{j,k} \to {\cX}_{K_{j}}$. By gluing, we obtain a morphism $\bar{\varphi}_j:U_{j,K_j}\to \cX$.

For any element $\sigma\in \Gal(K_j/K)$, consider the induced map $\sigma^*: U_{j,K_j}\to U_{j,K_j}$. Since the restriction of $\bar{\varphi}_j$ to the admissible open subset $U_{j,K_j}\setminus \{y_{j,1},\ldots,y_{j,n_j}\}$ is the base extension of $\varphi|_{U_j\setminus F\cap U_j}$ to $K_j$, we have 
$\bar{\varphi}_j\circ \sigma^*=\bar{\varphi}_j$
on this subset. Since this subset is Zariski dense and $\cX$ is separated, the equality $\bar{\varphi}_j\circ \sigma^*=\bar{\varphi}_j$ holds on $U_{j,K_j}$. By the descent of morphisms \cite[Corollary 4.2.5]{Con_ample}, we obtain morphisms $U_{j}\to \cX$. Since they coincide with each other on a Zariski dense subset of $U_{j}\cap U_{j+1}$ and $\cX$ is separated, they glue to define a morphism $D^\times_{K} \to \cX$. Again by assumption, it extends to a morphism $D_{K} \to \cX$. 
\end{proof}

\begin{thm}{(\cite[Theorem 1.1]{DL})}\label{PPness}
The eigencurve $\cC_{N}$ is Buzzard-Calegari proper at $\bQ_p$-algebraic points. 
\end{thm}

\begin{rmk}
Theorem \ref{PPness} is equivalent to the following statement: Suppose that we have a commutative diagram 
\[
\xymatrix{
D_K^\times \ar[r]\ar[d] & \cC_{N,K} \ar[d]\\
D_K \ar[r]& \cW_K
}
\]
for some finite extension $K/\bQ_p$, where $D^\times_K \to D_K$ is the natural inclusion and $\cC_{N,K} \to \cW_K$ is the weight map. Then there exists a morphism $D_K \to \cC_{N,K}$ which makes the diagram with this map added also commutative. This equivalence follows from the fact that the composite $D^\times_K \to \cC_{N,K}\to \cW_K$ automatically extends to $D_K$
(see for example \cite[Lemma 2.1]{DL}). 
\end{rmk}

\begin{rmk}
Originally, Buzzard-Calegari named the above property in the case of $K=\Cp$ the \textit{properness} of the eigencurve $\cC_{N,\Cp}$, and proved the case of $N=1$ and $p=2$ \cite[Theorem 8.1]{BuzCal}. Theorem \ref{PPness} follows from the conjectural properness of $\cC_{N,\Cp}$ in this sense. However, the authors do not know if they are equivalent.
\end{rmk}



\section{Old and new components}

Let $q$ be a prime dividing $N$. Put $N=N_0 q^m$ with $q\nmid N_0$ and $M=q^{-1}N$. We denote by $\frX_N^{q\text{-old}}$ the subset of $\frX_N$ consisting of classical normalized cuspidal eigenforms which are $q$-old. In this section, first we investigate a relation between irreducible components of $\cC_N$ and those of $\cC_M$. 

\begin{lem}\label{fincomp}
Let $\varphi:\cX\to \cY$ be any finite morphism of rigid analytic varieties. 
\begin{enumerate}
\item\label{fincompt} Let $C$ be any irreducible component of $\cY$ such that $\varphi^{-1}(C)\to C$ is surjective. Suppose that $\cY$ is reduced. Then there exists an irreducible component $C'$ of $\cX$ satisfying $\varphi(C')=C$. Moreover, if $C$ is equidimensional of dimension $d$, then so is $C'$. 
\item\label{fincomps} Let $C$ be any irreducible component of $\cX$. Suppose that $\cY$ and $C$ are both equidimensional of dimension $d$. Then the image $\varphi(C)$ is an irreducible component of $\cY$.
\end{enumerate}
\end{lem}
\begin{proof}
Let us prove (\ref{fincompt}). By \cite[Lemma 2.2.6]{Con_irr}, we can find a non-empty admissible open subset $U$ of $\cY$ which is contained in $C$. By \cite[Lemma 2.2.3]{Con_irr}, the Zariski closure of $U$ is $C$. We may assume that $U$ is affinoid. Since $\varphi^{-1}(U)$ is non-empty and quasi-compact, \cite[Paragraph below Definition 2.2.2]{Con_irr} implies that there exist finitely many irreducible components $X_1,\cdots,X_r$ of $\cX$ satisfying $\varphi^{-1}(U) \subseteq \bigcup_{i=1}^r X_i$. Thus we obtain $U\subseteq \bigcup_{i=1}^r \varphi(X_i)$ and $C\subseteq \bigcup_{i=1}^r \varphi(X_i)$, where every $\varphi(X_i)$ is an irreducible analytic subset of $\cY$. Since $C$ is irreducible, we have $C\subseteq \varphi(X_i)$ for some $i$ and \cite[Lemma 2.2.3]{Con_irr} implies $C=\varphi(X_i)$. 

Suppose moreover that $C$ is equidimensional of dimension $d$. We consider on $C$ and $C'=X_i$ the reduced structures. Then the induced map $\varphi: C'\to C$ is finite surjective and, over each admissible affinoid open subset $V\subset C$, the map $\varphi$ is defined by a finite injection of affinoid algebras $\cO(V)\to \cO(\varphi^{-1}(V))$. Thus the latter ring is of dimension $d$. Since $\dim(\cO_{C',x})$ is the same for any $x\in C'$ \cite[p. 496]{Con_irr}, we see that $C'$ is equidimensional of dimension $d$.

As for (\ref{fincomps}), we have a finite surjection $C\to \varphi(C)$, where we give both sides the reduced structures. As in the last part of the proof of (\ref{fincompt}), we see that $\varphi(C)$ is an irreducible analytic subset which is equidimensional of dimension $d$. Thus \cite[Corollary 2.2.7]{Con_irr} implies that $\varphi(C)$ is an irreducible component of $\cY$.
\end{proof}

\begin{dfn}
Let $C$ be any irreducible component of the eigencurve $\cC_N$ and $\kappa:C\to \cW$ the weight map. 
\begin{enumerate}
\item We say $C$ is of finite degree if $\deg(\kappa^{-1}(x))$ is finite for any $x\in \cW$. Otherwise, we say $C$ is of infinite degree.
\item We say $C$ is Eisenstein if $C\cap \frX_N^{\text{nc}}$ is Zariski dense in $C$. Otherwise, we say $C$ is cuspidal.
\item Suppose that $C$ is a cuspidal component. We say $C$ is $q$-old if $C\cap \frX_N^{q\text{-old}}$ is Zariski dense in $C$, and $C$ is $N$-new if it is not $q$-old for any prime $q\mid N$.
\end{enumerate}
\end{dfn}

Let $C\to Y$ be a finite surjection to a Fredholm hypersurface over a component $\cW^0$ of $\cW$. Then $C$ is of finite degree if and only if so is $Y$. Indeed, if $Y$ is of infinite degree, then there exists $x\in \cW^0$ such that the fiber of $Y\to \cW$ at $x$ is an infinite set, which implies that $C$ is not of finite degree. In particular, if $C$ is of infinite degree, then some fiber of $\kappa: C\to \cW$ is an infinite set.

If $C$ is Eisenstein, then as in \cite[\S3.6]{CM} we see that $C$ is in the ordinary locus of $\cC_N$ and hence it is finite over $\cW$. If $C$ is cuspidal, then it is contained in the cuspidal eigencurve $\cC_N^0$ and thus $C\cap \frX^0_N$ is Zariski dense in $C$. Since the Galois representation attached to any Eisenstein series is reducible, \cite[\S4.2]{Che_Psc} implies that for any cuspidal component $C$, the Zariski closure of $C\cap (\frX_N\setminus \frX_N^0)$ is a discrete subset of $C$.

\begin{lem}\label{qTU}
Let $C$ be any irreducible component of $\cC_N$ which is $q$-old. Then there exist an irreducible component $C'$ of $\cC_M$ and finite surjections over $\cW$
\[
C' \gets C''\to C
\]
from a rigid analytic variety $C''$.
\end{lem}
\begin{proof}
Let $\kappa(q)$ be the image of $q\in \bZ_p^\times$ by the restriction of the universal character
\[
\bZ_p^\times \to \cO(\cW)^\times \to \cO(\cC_M)^\times.
\]
We denote by $\tilde{\cC}_M$ the closed subvariety of $\cC_M\times \bA^1$ defined by 
\[
\left\{
\begin{array}{ll}
U^2-T_q U+ q^{-1}\kappa(q)\langle q\rangle_M & (q\nmid M) \\
U^2-U_q U & (q\mid M),
\end{array}
\right.
\]
where $U$ is the parameter of $\bA^1$. The first projection $\pi':\tilde{\cC}_M\to \cC_M$ is finite flat. Put $\tilde{\frX}_M^0=(\pi')^{-1}(\frX_M^0)$.

Consider the diagram
\[
\xymatrix{
\tilde{\cC}_M\ar[r] &\cC_{M}\times \bA^1 \ar[r]^-j & X_M \times \Gm \times \prod_{l|N_0} \bA^1 \times \bA^1 \ar[d]^{\pi} \\
 \cC_{N} \ar[rr]& & X_N \times \Gm \times \prod_{l|N_0} \bA^1 \times \bA^1,
}
\]
where the last $\bA^1$'s on the rightmost entries are the $q$-parts, the map $j$ is given by
\[
j(f,u)=(\rho_f,a_p(f)^{-1}, (a_l(f))_{l\mid N_0}, u),
\]
the other horizontal arrows are the natural closed immersions and the map $\pi$ is induced by the natural map $G_{\bQ,\Sigma_{pN}}\to G_{\bQ, \Sigma_{pM}}$. Note that the morphism $X_M\to X_N$ is over $\cW$.

From \cite[Corollary 4.6.20]{Miy} we see that, for any element $g\in \frX_M^0$, the subset of $f\in\frX_N^{q\text{-old}}$ satisfying $a_l(f)=a_l(g)$ for any prime $l\neq q$ is in bijection with the fiber $(\pi')^{-1}(g)$ via the map $f\mapsto a_q(f)$. 
Moreover, every element of $\frX_N^{q\text{-old}}$ is obtained in this way from some $g\in \frX_M^0$. This means $(\pi\circ j)(\tilde{\frX}^0_M)= \frX_N^{q\text{-old}}$. Let $\tilde{C}$ be the Zariski closure of $\tilde{\frX}^0_M$ in $\tilde{\cC}_M$. 
We give $C$ and $\tilde{C}$ the reduced structures. Then $\pi$ induces a morphism $\tilde{C}\to \cC_N$, which we also denote by $\pi$. Since $\cC_M$ and $\cC_N$ are finite over $\cW\times \Gm$, the morphism $\pi:\tilde{C}\to \cC_N$ is also finite and, by the assumption that $C$ is $q$-old, the map $\pi^{-1}(C)\to C$ is surjective.

Now Lemma \ref{fincomp} (\ref{fincompt}) shows that there exists an irreducible component $C''$ of $\tilde{C}$ which is equidimensional of dimension one satisfying $\pi(C'')=C$. From Lemma \ref{fincomp} (\ref{fincomps}), we see that $C'=\pi'(C'')$ is an irreducible component of $\cC_M$. This concludes the proof.
\end{proof}

For $N$-new components, we show the following density result of $N$-newforms. 

\begin{lem}\label{qnewdense}
Let $C$ be any irreducible component of $\cC_N$ which is $N$-new. Let $f$ be an element of $C$ of integral weight $k$, level $\Gamma_1(Np)$ and slope $s$. Then, for any sufficiently large integer $M$, there exists a classical $N$-new eigenform $g\in C\cap \frX_N$ of integral weight $k'> 2s+1$, level $\Gamma_1(Np)$ and slope $s$ satisfying 
$k\equiv k' \bmod (p-1)p^M$.
\end{lem}
\begin{proof}
Put $x=\kappa(f)$. By assumption, for any prime $q\mid N$ there exists an analytic subset $F_q$ of $C$ satisfying 
\[
C\cap\frX_N^{q\text{-old}}\subseteq F_q \subsetneq C.
\]
Consider the decomposition $F_q=\bigcup_{i\in I} F_{q,i}$ into irreducible components. If $F_{q,i}$ is equidimensional of dimension one, then \cite[Corollary 2.2.7]{Con_irr} yields $C=F_{q,i}$, which contradicts $C\neq F_q$. Therefore, $F_{q,i}$ and $F_q$ are discrete subsets. Moreover, we have a discrete subset $E$ of $C$ which contains all points corresponding to classical Eisenstein series.
Hence $F=E\cup \bigcup_{q\mid N}F_q$ is also a discrete subset of $C$.

Now let us consider the projection $\pi_1: \cC_N\to Z_1$ to the spectral curve $Z_1$ for the operator $U_p$, and the projection $\mu:Z_1\to \cW$. Let $Y$ be the image of $C$ by the map $\pi_1$ with the reduced structure. 
Since $\pi_1$ is finite, $Y$ is irreducible.
By \cite[Theorem 4.6]{Buz} and \cite[Proposition 6.3.2]{Che_GLn}, there is an affinoid open neighborhood $U$ of $\pi_1(f)$ in $Z_1$ with $\mu(U)=V\subseteq \cW$ connected affinoid open such that $U$ is finite over $V$ and it is the Fredholm hypersurface defined by a polynomial factor of the characteristic power series of the $U_p$-action on the space of overconvergent modular forms over $V$. The intersection $Y\cap U$ is a union of irreducible components of $U$, and hence it is again a Fredholm hypersurface defined by a polynomial, which we denote by
\[
Q(T)=1+ a_1(w)T+\cdots+a_n(w) T^n.
\]
Since $\mu:U\to V$ is finite and $V$ is irreducible, we have $\mu(Y\cap U)=V$.

Let $\cW^0=\Spf(\bZ_p[[w]])^\rig$ be the component of $\cW$ to which $C$ is mapped. Put $\mathbf{p}=|2|_p^{-1}p$ and $x_k=(1+\mathbf{p})^k$. Note that the point $x\in \cW^0$ is defined by $w=x_k$, and we have $V\subseteq \cW^0$.
We denote by $Q_k(T)$ the specialization of $Q(T)$ at $w=x_k$. Then the image $\pi_1(f)\in Z_1$ gives a root of $Q_k(T)$ of $p$-adic valuation $-s$.

Let $(m_s, \lambda_s)$ be the right endpoint of the segment of slope $s$ in the Newton polygon of $Q_k(T)$.
Note that, for any positive integers $M$ and $t$ with $p\nmid t$, we have
\[
v_p((1+\mathbf{p})^{p^M t}-1)= M+v_p(\mathbf{p}).
\]
Thus there exists an integer $M'$ such that, if an integer $k'$ satisfies $k\equiv k'\bmod (p-1)p^{M}$ for some $M>M'$, then we have 
\[
v_p(x_k -x_{k'}) >\lambda_s+n s.
\]
This yields 
\[
\min\{ v_p(a_l(x_k)), \lambda_s\}=\min\{ v_p(a_l(x_{k'})), \lambda_s\}
\]
for any $l\leq m_s$ and $v_p(a_l(x_{k'}))>\lambda_s+(l-m_s) s$ for any $l> m_s$. Hence the Newton polygons of $Q_k(T)$ and $Q_{k'}(T)$ coincide with each other on the segments of slopes no more than $s$.

Take a closed disc $D$ centered at $x_k$ of some radius $\rho$ sufficiently small so that $D$ is an admissible open subset of $V$. Let $M$ be any integer satisfying $M > \max\{-\log_p \rho,M'\}$. Put
\[
D'=\{w \in D\mid |w-x_k|_p=p^{-(M+v_p(\mathbf{p}))} \}.
\]
Since $U'=Y\cap U\cap (D'\times B[0,p^{s}])$ is affinoid and $\pi_1(F)$ is a discrete subset of $Y$, the intersection $U'\cap \pi_1(F)$ is a finite set. Thus we can find an integer $k'>2s+1$ of the form $k'=k+t(p-1)p^M$ with $p\nmid t$ such that $x_{k'}$ is contained in the set 
\[
D'\setminus \mu(U'\cap \pi_1(F)).
\] 
Since $Q_{k'}(T)$ has roots of $p$-adic valuation $-s$, we can take $y'\in Y\cap U$ in the fiber of $x_{k'}$ corresponding to such a root. Since $y'\in U'$, the choice of $x_{k'}$ implies $y'\notin \pi_1(F)$. By Coleman's classicality theorem \cite[Theorem 1.1]{Col_cl2}, any point $g\in C$ above $y'$ is a classical eigenform of level $\Gamma_1(Np)$. Since $y'\notin \pi_1(F)$, the eigenform $g$ is $N$-new. This concludes the proof.
\end{proof}



\section{Proof of the main theorem}

In this section, we prove Theorem \ref{main}.
Let $\cW^0=\Spf(\bZ_p[[w]])^\rig$ be the component of $\cW$ containing the image of $C$ for the weight map $\kappa:\cC_N\to \cW$. Since every Eisenstein component is finite over $\cW$, we may assume that $C$ is cuspidal.

First we assume that $C$ is $N$-new. In this case, we begin with following the proof of \cite[Lemma 7.4.3]{CM}: Since the image $\kappa(C)$ only omits finitely many points of $\cW^0$, there exists $f\in C$ of some integral weight and level $\Gamma_1(Np)$. By Lemma \ref{qnewdense}, we can find an $N$-new classical eigenform $g\in C$ of weight $k$, level $\Gamma_1(Np)$ and slope $s< (k-1)/2$.

We denote by $M_k(\Gamma_1(Np))^{\leq s}$ the space of classical modular forms of weight $k$, level $\Gamma_1(Np)$ and slope no more than $s$.
By the classicality theorem, it is equal to the space of overconvergent modular forms satisfying the same conditions. By the choice of $s$, the operator $U_p$ acts semi-simply on $M_k(\Gamma_1(Np))^{\leq s}$.
Then, as in the proof of \cite[Sublemma 6.2.3]{CM}, there exists $\alpha\in 1+p\cH$ such that
the $\alpha U_p$-eigenspace of $M_k(\Gamma_1(Np))^{\leq s}$ containing $g$ does not contain eigenforms other than $g$. Furthermore, the generalized $\alpha U_p$-eigenspace of $M_k(\Gamma_1(Np))^{\leq s}$ containing $g$ is one-dimensional; otherwise its $N$-old part or Eisenstein part would be a non-trivial $\cH$-stable subspace and thus contain an eigenform other than $g$. Consider the projection $\pi_\alpha: \cC_N\to Z_\alpha$ to the spectral curve for $\alpha U_p$.
Then the image $\pi_{\alpha}(g)\in Z_{\alpha}$ gives a simple root of the characteristic power series of $\alpha U_p$ specialized at weight $k$. This means that $\pi_\alpha$ maps $C$ onto an irreducible component $Y$ of $Z_\alpha$ without multiplicity. Note that $Y$ is also of finite degree.

Let $F$ be any finite subset of $\cW^0$ such that the map $\mu:Y\setminus \mu^{-1}(F) \to \cW^0\setminus F$ is finite.

\begin{lem}\label{genisom}
There exists a discrete subset $E$ of $Y\setminus \mu^{-1}(F)$ such that the map
\[
C|_{(Y\setminus \mu^{-1}(F) )\setminus E} \to (Y\setminus \mu^{-1}(F) )\setminus E
\]
induced by $\pi_\alpha$ is generically isomorphic. Namely, there exists a discrete subset $E'$ of $(Y\setminus \mu^{-1}(F))\setminus E$ such that the map
\[
C|_{((Y\setminus \mu^{-1}(F) )\setminus E)\setminus E'}\to ((Y\setminus \mu^{-1}(F) )\setminus E)\setminus E'
\]
is an isomorphism.
\end{lem}
\begin{proof}
We denote by $P(T)$ the characteristic power series of $\alpha U_p$ over $\cW^0$ and by
\[
Q(T)=1+a_1(w)T+\cdots+ a_n(w) T^n
\]
the irreducible factor of $P(T)$ defining $Y$. In the ring of entire series $\bZ_p[[w]]\{\{T\}\}$ over $\bZ_p[[w]]$, we can write $P(T)=Q(T)H(T)$, where $Q(T)$ and $H(T)$ have no common factor. For any admissible affinoid open subset $W\subseteq \cW^0$, we denote the restriction of $Q(T)$ to $W$ by $Q_W(T)$.
For any $x\in \cW^0$, we denote the specialization of $Q(T)$ at $x$ by $Q_x(T)$, and similarly for $H(T)$. Note that if $W$ is an admissible affinoid open subset of $\cW^0\setminus F$, then the element $a_n(w)$ is invertible on $W$.

Put $E_0=\bigcup_{Y'\neq Y} Y\cap Y'$, where $Y'$ runs over the set of irreducible components of $Z_\alpha$ other than $Y$. By \cite[Corollary 2.2.7]{Con_irr}, each $Y\cap Y'$ is a discrete subset of $Y$. On each admissible affinoid open subset $U$ of $Y$, the intersection $U\cap E_0$ is a discrete subset since $U$ meets only finitely many $Y'$'s. Thus $E_0$ is also a discrete subset of $Y$.

\begin{lem}
The subset $G=\{x\in \cW^0\setminus F\mid (Q_x(T),H_x(T))\neq 1\}$ of $\cW^0\setminus F$ is discrete.
\end{lem}
\begin{proof}
It is enough to show that for any admissible affinoid open subset $W$ of $\cW^0\setminus F$, the intersection $W\cap G$ is a discrete subset of $W$. 
We claim that $W\cap G=\mu(Y|_W \cap E_0)$. Indeed, take $x\in W\cap G$. We have $(Q_x(T),H_x(T))\neq 1$ in the ring $K(x)\{\{T\}\}$, where $K(x)$ is the residue field of $x$. From \cite[Lemma 4.1.1 (1)]{Con_irr}, we see that the natural maps
\[
\xymatrix{
K(x)[T]/(Q_x(T)) \ar[rd]\ar[d] & \\
K(x)\{\{T\}\}/(Q_x(T))\ar[r]& K(x)\langle p^m T \rangle/(Q_x(T))
}
\]
are isomorphic for any sufficiently large $m$. Therefore,
the natural map
\[
K(x)\{\{T\}\}/(Q_x(T),H_x(T))\to K(x)\langle p^m T \rangle/(Q_x(T),H_x(T))
\]
is also an isomorphism for any sufficiently large $m$. 
Thus, for the closed ball $B[0,p^m]\subseteq \bA^1$ of sufficiently large radius $p^m$ centered at the origin, there exists $y\in (W\times B[0,p^m])\cap Y$ with $\mu(y)=x$ such that $y$ is also contained in the analytic subset of $\cW^0\times \bA^1$ defined by $H(T)$. By \cite[Theorem 4.3.2]{Con_irr}, the latter is the union of irreducible components of $Z_\alpha|_{\cW^0}$ other than $Y$. This implies $y\in E_0$ and $x\in \mu(Y|_W\cap E_0)$. Conversely, take $y\in Y\cap E_0$ satisfying $x=\mu(y)\in W$. Then $y\in W\times B[0,p^m]$ for a sufficiently large $m$. The point $y$ defines a maximal ideal of the ring $K(x)\langle p^m T \rangle/(Q_x(T),H_x(T))$. Hence we have $(Q_x(T),H_x(T))\neq 1$ in the subring $K(x)\{\{T\}\}$ of $K(x)\langle p^m T \rangle$ and the claim follows.
Since $\mu:Y|_W\to W$ is finite,  $\mu(Y|_W \cap E_0)$ is a discrete subset of $W$.
\end{proof}

Put $E=\mu^{-1}(G)$. Since the map $\mu: Y\setminus \mu^{-1}(F)\to \cW^0\setminus F$ is finite, the subset $E$ of $Y\setminus \mu^{-1}(F)$ is discrete. 

Consider any admissible affinoid open subset $W$ of $(\cW^0\setminus F)\setminus G$. It is also an admissible affinoid open subset of $\cW^0$. Put $\tilde{Q}_W(T)=a_n(w)^{-1}Q_W(T)$. From the definition of the resultant (see \cite[\S A3]{Col_pBan}), we see that $\mathrm{Res}(\tilde{Q}_W,H_W)_x=\mathrm{Res}(\tilde{Q}_x,H_x)$ for any $x\in W$, where $\tilde{Q}_x$ is the specialization of $\tilde{Q}_W$ at $x$. By \cite[Lemma A3.7]{Col_pBan}, we have $(Q_W(T),H_W(T))=1$. Thus the factor $Q_W(T)$ gives an admissible affinoid open subset $V$ of $Z_\alpha|_{W}$ which is a member of its canonical admissible covering (see for example \cite[Proposition 6.3.2]{Che_GLn}). In fact, we have $V=Y|_W$, since $V$ is a clopen subset of $Z_\alpha|_W$ defined by $Q_W(T)$.

Since $Y$ is without multiplicity, $V$ is reduced. Since the proof of \cite[Proposition 7.1.3]{CM} is valid also in our situation, the projection $\cD_N|_V \to V=Y|_W$ is generically isomorphic. 
Note that $\cD_N|_V$ and $V$ are finite flat of the same degree over $W$ and the locus where this projection is not isomorphic is the analytic subset defined by the determinant of the injection $\cO(V)\to \cO(\cD_N|_V)$. 
By gluing, we obtain a discrete subset $E'$ of $(Y\setminus \mu^{-1}(F))\setminus E$ such that the map 
\[
\cD_N|_{((Y\setminus \mu^{-1}(F))\setminus E)\setminus E'} \to ((Y\setminus \mu^{-1}(F))\setminus E)\setminus E'
\]
is an isomorphism. 

Since $C$ is an irreducible component of $\cD_{N,\red}\simeq \cC_N$, we have the commutative diagram
\[
\xymatrix{
\cD_N|_{((Y\setminus \mu^{-1}(F))\setminus E)\setminus E'} \ar[r]^-{\sim} & ((Y\setminus \mu^{-1}(F))\setminus E)\setminus E'\\
C|_{((Y\setminus \mu^{-1}(F))\setminus E)\setminus E'}, \ar[ru]\ar[u] &
}
\]
where the vertical arrow is the natural closed immersion and the oblique arrow is a surjection between reduced rigid analytic varieties. Hence the oblique arrow is an isomorphism.
\end{proof}

Let $K/\bQ_p$ be any finite extension and take any $x\in \cW^0(K)$. By Proposition \ref{Lutke}, there exists a closed disc $D$ of sufficiently small radius in $|K^\times|_p$ centered at $x$ such that $D$ is an admissible open subset of $\cW^0$, $D^\times=D\setminus \{x\}\subseteq \cW^0\setminus F$ and for some finite extension $L/K$, we have an isomorphism
\[
\coprod_{i=1,\ldots,M} D_{i,L}^\times \to Y_L |_{D_L^\times}
\]
over $D_L^\times$, where $D_{i,L}$ is a closed disc of some radius over $L$ centered at the origin $O$, $D^\times_{i,L}=D_{i,L}\setminus \{O\}$ and the map $D_{i,L}^\times \to D_L^\times$ is given by $z\mapsto x+z^{m_i}$ for some positive integer $m_i$.

Consider any component $Y_{L,i}^\times\simeq D^\times_{i,L}$ of $Y_L|_{D^\times_L}$. By Lemma \ref{genisom}, there exist discrete subsets $E_i\subseteq D^\times_{i,L}$ and $E'_i \subseteq D^\times_{i,L}\setminus E_i$ such that the map $\pi_{\alpha,L}: C_L \to Y_L$ induces an isomorphism 
\[
C_L|_{(D^\times_{i,L} \setminus E_i)\setminus E'_i} \simeq (D^\times_{i,L} \setminus E_i)\setminus E'_i.
\]
Thus we have a section 
\[
s_i:(D^\times_{i,L}\setminus E_i)\setminus E'_i \to C_L 
\]
of $\pi_{\alpha,L}$. 

We take a covering of $D_{i,L}^\times$ adapted to the discrete subset $E_i$, as in the proof of Lemma \ref{PPdiscr}. Namely, we take an admissible covering $D^\times_{i,L}=\bigcup_{j\in \bZ_{\geq 0}} U_{i,j}$ consisting of
\[
{U}_{i, j}=\{ y\in D_{i,L} \mid \rho'_{i,j} \leq|y|_p \leq \rho_{i,j}\} 
\]
for some positive real numbers $\rho_{i,j},\rho'_{i,j} \in |L^\times|_p$. Take a finite Galois extension $L_{i,j}/L$ such that every element of
\[
(E_i\cap U_{i,j})_{L_{i,j}}=\{y_{i,j,1},\ldots, y_{i,j,n_{i,j}}\}
\]
is $L_{i,j}$-rational. Take mutually disjoint sufficiently small closed discs
\[
D_{i,j,k}=\{y\in D_{i,L_{i,j}}\mid |y-y_{i,j,k}|_p \leq r_{i,j}\}
\]
with $r_{i,j}\in |L_{i,j}^\times|_p$ and put
\[
V_{i,j}=\{ y\in {U}_{i,j,L_{i,j}}\mid |y-y_{i,j,k}|_p\geq s_{i,j}\text{ for any }k\}
\]
for some $s_{i,j}\in |L_{i,j}^\times|_p$ satisfying $s_{i,j}<r_{i,j}$.

On each small punctured disc $D_{i,j,k}^\times=D_{i,j,k}\setminus \{y_{i,j,k}\}$, the subset $E'_i\cap D_{i,j,k}^\times$ is discrete. By Lemma \ref{PPdiscr} and Theorem \ref{PPness}, the morphism $s_i$ restricted to $D_{i,j,k}^\times\setminus (E'_i\cap D_{i,j,k}^\times)$ extends to a morphism $g_{i,j,k}: D_{i,j,k}\to \cC_{N,L_{i,j}}$. Since the set $D_{i,j,k}^\times\setminus (E'_i\cap D_{i,j,k}^\times)$ is Zariski dense in $D_{i,j,k}$, we have $g_{i,j,k}(D_{i,j,k})\subseteq C_{L_{i,j}}$ and since $D_{i,j,k}$ is reduced, it really factors through $D_{i,j,k}\to C_{L_{i,j}}$. 

On each admissible affinoid open subset $V_{i,j}$, the intersection $E'_i\cap V_{i,j}$ is finite. Thus, again taking a sufficiently small closed disc centered at each element of $E'_i\cap V_{i,j}$ and applying Theorem \ref{PPness} as in the proof of Lemma \ref{PPdiscr}, we obtain an extension $V_{i,j}\to C_{L_{i,j}}$ of the section $s_i$.

These extended maps coincide with each other on a Zariski dense subset of $D_{i,j,k}\cap V_{i,j}$. Thus, by gluing and the Galois descent as before, they define morphisms $U_{i,j}\to C_L$, which also glue to yield a morphism $D^\times_{i,L} \to C_L$. Again by Theorem \ref{PPness}, it extends to a morphism $g_i:D_{i,L} \to C_L$ such that its restriction to $(D^\times_{i,L}\setminus E_i)\setminus E'_i$ is equal to the section $s_i$. 
Taking the direct sum of $g_i$'s, we obtain a morphism 
\[
g: S=\coprod_{i=1,\ldots,M}D_{i,L} \to C_L.
\]

On the $i$-th component $D_{i,L}$ of $S$, the composite 
\[
h_i=\kappa\circ g_i:D_{i,L} \to C_L \to \cW^0_L
\] 
coincides with the map $z\mapsto x+z^{m_i}$ on $(D^\times_{i,L}\setminus E_i)\setminus E'_i$. Since $D_{i,L}$ is reduced and $\cW^0_L$ is separated, the Zariski density of $(D^\times_{i,L}\setminus E_i)\setminus E'_i \subseteq D_{i,L}$ implies that $h_i$ itself coincides with this  map. 
Thus it factors through $D_L\subseteq \cW_L^0$ and we obtain a morphism $g: S\to C_L|_{D_L}$. Moreover, $S$ is finite over $D_L$. Similarly, since $(D_{i,L}^\times\setminus E_i)\setminus E'_i\subseteq D_{i,L}^\times$ is also Zariski dense, we see that the composite 
\[
D^\times_{i,L}\subseteq D_{i,L} \to C_L \to Y_L 
\]
induces the isomorphism $D_{i,L}^\times \simeq Y_{L,i}^\times$.
Therefore, the image of the composite 
\[
\pi_{\alpha,L}\circ g: S\to C_L|_{D_L}\to Y_L|_{D_L}
\]
is the complement in $Y_L|_{D_L}$ of a finite subset of the fiber over $x$. 
By Lemma \ref{Yfinite}, the map $Y_L|_{D_L}\to D_L$ is finite. Hence we conclude that the map $Y \to \cW^0$ is finite. Since $C$ is finite over $Y$, the theorem follows for $N$-new components and the case $N=1$.

Now we proceed by induction on $N$. 
Suppose that the theorem is proved for any tame level less than $N$. To show the theorem for the case of tame level $N$, we may assume that $C$ is $q$-old for some prime $q\mid N$. Put $M=q^{-1}N$. 
Lemma \ref{qTU} implies that there exist an irreducible component $C'$ of $\cC_M$ and a rigid analytic variety $C''$ with finite surjections 
\[
C'\gets C''\to C
\]
over $\cW$. Since $C$ is of finite degree, $C''$ is also of finite degree over $\cW$. Thus $C'$ is also of finite degree.
By the induction hypothesis, $C'$ is finite over $\cW$ and so is $C''$. Let $Y$ be the image of $C$ by the projection $\pi_1:\cC_N\to Z_1$ with the reduced structure. Note that $Y$ is also of finite degree. Since the composite $C'' \to C \to Y$ is surjective, Lemma \ref{Yfinite} implies that $Y$ is finite over $\cW$. Thus $C$ is also finite over $\cW$.
This concludes the proof of Theorem \ref{main}. \qed



\section{Application}\label{SecApp}

Chenevier proved that an irreducible component of an eigencurve which is finite over the weight space is in the ordinary locus if we know that the slopes tend to zero near the boundary of the weight space (\cite[\S 3.7]{Che}. See also \cite[Proposition 3.24]{LWX}). The latter claim on the boundary behavior is a conjecture of Coleman-Mazur-Buzzard-Kilford (see \cite[Conjecture 1.2]{LWX}), and it is shown for the case of $N=1$, $p=2$ \cite[Theorem A]{BuzKil} and for quaternionic eigencurves \cite[Theorem 1.3]{LWX}. 
Moreover, for the case of $N=1$ and $p=3$, it is also shown on the even component of the weight space, namely the component where the weight characters $\kappa$ satisfy $\kappa(-1)=1$ \cite[Theorem 1.1]{Roe}. 
Thus, by combining Theorem \ref{main} with those results (and for quaternionic eigencurves, also with the $p$-adic Jacquet-Langlands correspondence \cite{Che_JL}), we obtain the following corollaries.

\begin{cor}\label{CorBK}
For $N=1$ and $p=2$, the only irreducible component of the eigencurve $\cC_1$ of finite degree is the Eisenstein component.
\end{cor}

\begin{cor}\label{CorRoe}
For $N=1$ and $p=3$, the only irreducible component of the eigencurve $\cC_1$ of finite degree over the even component of the weight space is the Eisenstein component.
\end{cor}

\begin{cor}\label{CorLWX}
Let $D$ be a quaternion algebra over $\bQ$ of discriminant $d$ satisfying $(Np,d)=1$. Let $\cC^D_N$ be the eigencurve of tame level $N$ for overconvergent modular forms on $D^\times$. Then every irreducible component of $\cC^D_N$ of finite degree is ordinary.
\end{cor}





\end{document}